\documentclass[a4paper,10pt]{article}

\usepackage{geometry}
\usepackage{amsmath}
\usepackage{amssymb}
\usepackage{amsthm}
\usepackage{mathrsfs}
\usepackage{color}
\usepackage{amscd}
\usepackage{graphicx}
\usepackage{hyperref}

\newtheorem{theorem}{Theorem}%[section]
\newtheorem*{theorem*}{Theorem}
\newtheorem{definition}[theorem]{Definition}

\newtheorem{proposition}[theorem]{Proposition}

\theoremstyle{definition}
\newtheorem{remark}[theorem]{Remark}

\newcommand{\BB}{\mbb{B}}
\newcommand{\C}{\mbb{C}}
\newcommand{\HH}{\mbb{H}}

\newcommand{\I}{\mc{I}}

\newcommand{\R}{\mbb{R}}
\newcommand{\OO}{\Omega}

\newcommand{\mbb}{\mathbb}
\newcommand{\mc}{\mathcal}

\newcommand{\lra}{\to}

\newcommand{\ui}{\imath}

\newcommand{\dibar}{\overline\partial}
\newcommand{\dbcrf}{\dibar_{\scriptscriptstyle CRF}}
\newcommand{\dcrf}{\partial_{\scriptscriptstyle CRF}}
\newcommand\sd[1]{{#1}'_s}
\def\SS{\mathbb S}
\newcommand\Hc{\HH\otimes_{\R}{\C}}

\newcommand\Sl{\mathcal S}

\newcommand\dd[2]{\dfrac{\partial#1}{\partial#2}} 
\newcommand\ddd[2]{\frac{\partial#1}{\partial#2}}

\newcommand\IM{\operatorname{Im}}
\newcommand\RE{\operatorname{Re}}

\newcommand\ord{\operatorname{ord}}

\newcommand\Reg{\mathcal {R}}
\newcommand\SR{\Sl\Reg}

\newcommand{\bc}{\begin{center}}
\newcommand{\ec}{\end{center}}
\newcommand\vs[1]{{#1}_s^\circ}
\newcommand{\CRF}{\scriptscriptstyle CRF}

\begin{document}
\title{\bf A four dimensional Jensen formula}

\author{A.~Perotti\footnote{Partially supported by GNSAGA of INdAM}
\\
\small Dipartimento di Matematica, Universit\`a di Trento\\ 
\small Via Sommarive 14, I-38123 Povo Trento, Italy\\
\small alessandro.perotti@unitn.it}

\date{  }

\maketitle

\begin{abstract}
We prove a Jensen formula for slice-regular functions of one quaternionic variable. The formula relates the value of the function and of its first two derivatives at a point with its integral mean on a three dimensional sphere centred at that point and with the disposition of its zeros. The formula can be extended to semiregular slice functions.
\\


\end{abstract}

\section{Introduction and preliminaries}

The aim of this note is to prove a Jensen formula for slice-regular functions of one quaternionic variable. We show how the results obtained in \cite{Pe2018} can be applied to extend to any slice-regular function the formula proved by Altavilla and Bisi in \cite{AltavillaBisi} for slice-preserving functions.
The formula relates the value of the function and of its first and second derivatives at a point on the real axis, its integral mean on a three dimensional sphere centred at that point, and the position of its zeros. The formula can be generalized to semiregular slice functions, where also poles enter in the formula.
See \cite[Theorem 5.2]{EntireSRF} for another Jensen-type formula for slice-regular functions, in which the integration is performed on a two-dimensional slice of the domain.

Slice-regular functions constitute a recent function theory in several hypercomplex settings %, including quaternions and real Clifford algebras 
(see \cite{GS, %CSS,
GhPe_AIM, GSS,AlgebraSliceFunctions}). 
This class of functions was introduced by Gentili and Struppa \cite{GS} for functions of one quaternionic variable. Let $\HH$ denote the skew field of quaternions, with basic elements $i,j,k$. For each quaternion $J$ in the sphere 
 \[\SS=\{J\in\HH\ |\ J^2=-1\}=\{x_1i+x_2j+x_3k\in\HH\ |\ x_1^2+x_2^2+x_3^2=1\}\]
of imaginary units, let $\C_J=\langle 1,J\rangle\simeq\C$ be the subalgebra generated by $J$. Then we have the ``slice'' decomposition
\[\HH=\bigcup_{J\in \SS}\C_J, \quad\text{with $\C_J\cap\C_K=\R$\quad for every $J,K\in\SS,\ J\ne\pm K$.}\]

A differentiable function $f:\OO\subseteq\HH\rightarrow\HH$  is called \emph{(left) slice-regular} on $\OO$ if, for each $J\in\SS$, the restriction
\[f_{\,|\OO\cap\C_J}\ : \ \OO\cap\C_J\rightarrow \HH\]
is holomorphic with respect to the complex structure defined by left multiplication by $J$. For example, polynomials $f(x)=\sum_m x^ma_m$ with quaternionic coefficients on the right are slice-regular on $\HH$ and convergent power series are slice-regular on an open ball centered at the origin.

Let $x_0,x_1,x_2,x_3$ denote the real components of a quaternion $x=x_0+x_1i+x_2j+x_3k$.
In the following, we use the *-algebra structure of $\HH$ given by the $\R$-linear antiinvolution $x\mapsto \bar x =x_0-x_1i-x_2j-x_3k$. Let $t(x):=x+\bar x=2\RE(x)$ be the \emph{trace} of $x$ and $n(x):=x\bar x=x_0^2+x_1^2+x_2^2+x_3^2=|x|^2$  the \emph{norm} of $x$. We also set $\IM(x)=(x-\bar x)/2=x_1i+x_2j+x_3k$.

Let $\HH\otimes_{\R}\C$ be the algebra of complex quaternions, with elements $w=a+\ui b$, $a,b\in\HH$, $\ui^2=-1$.
Every quaternionic polynomial $f(x)=\sum_m x^ma_m$ lifts to a unique polynomial function $F:\C\rightarrow\HH\otimes_{\R}\C$  which makes the following diagram commutative for every $J \in \SS$:
\[
\begin{CD}
\C\simeq \R\otimes_\R\C @>F> >\HH\otimes_\R\C\\ % right arrow with labels
@V \Phi_J V  V %down arrow with labels
@V V \Phi_J V\\%down arrow with labels
\HH @>f> >\HH % right arrow with labels
\end{CD} 
\]
where $\Phi_J: \HH\otimes_{\R}\C \to \HH$ is defined by $\Phi_J(a+\ui b):=a+Jb$. The lifted polynomial is simply $F(z)=\sum_m z^ma_m$, with variable $z=\alpha+\ui \beta\in\C$.

In this lifting, the usual product of polynomials with coefficients in $\HH$ on one fixed side (the one obtained by imposing that the indeterminate commutes with the coefficients when two polynomials are multiplied together) corresponds to the pointwise pro\-duct in the algebra $\HH\otimes_{\R}\C$.

The remark made above about quaternionic polynomials suggests a way to define $\HH$-valued functions on a class of open domains in $\HH$.
Let $D\subseteq\C$ be a set that is invariant with respect to complex conjugation. 
In $\Hc$ consider the complex conjugation that maps $w=a+\ui b$ to $\overline w=a-\ui b$ ($a,b\in \HH$).
If a function $F: D \lra \Hc$ satisfies  $F(\overline z)=\overline{F(z)}$ for every $z\in D$, then $F$  is called a \emph{stem function} on $D$. Let $\OO_D$ be the \emph{circular} subset of $\HH$ defined by 
\[\OO_D=\bigcup_{J\in\SS}\Phi_J(D).\]
 The stem function $F=F_1+\ui F_2:D \lra \Hc$  induces the \emph{(left) slice function} $f=\I(F):\OO_D \lra \HH$ in the following way: if $x=\alpha+J\beta =\Phi_J(z)\in \OO_D\cap \C_J$, then  
\[ f(x)=F_1(z)+JF_2(z),\]
where $z=\alpha+\ui\beta$. 

The previous lifting suggests also the definition of the \emph{slice product} of two slice functions $f=\I(F)$ and $g=\I(G)$. It is the slice function $f\cdot g=\I(FG)$ obtained  by means of the product in the algebra $\HH\otimes_{\R}\C$.
We recall the formula that links the slice product to the quaternionic pointwise product: if $f(x)=0$, then $(f\cdot g)(x)=0$, while  for every $x$ such that $f(x)\ne0$ it holds $(f\cdot g)(x)=f(x)g(f(x)^{-1}xf(x))$.

The function $f=\I(F)$ is called \emph{slice-preserving} if $F_1$ and $F_2$ are real-valued (this is the case already considered by Fueter \cite{F} for holomorphic $F$). 
In this case, $f(\overline x)=\overline{f(x)}$ for each $x\in\OO_D$, and the slice product $f\cdot g$ coincides with the pointwise product of $f$ and $g$ for any slice function $g$.

The slice function $f$ is called \emph{circular} if $F_2\equiv0$. In this case, if $x=\alpha+\beta J\in\HH\setminus\R$, $f(y)=f(x)$ for every $y$ in the sphere $\SS_x=\alpha+\beta\SS$.

If the stem function $F$ is holomorphic, the slice function $f=\I(F)$ is called \emph{(left) slice-regular}. We shall denote by $\SR(\OO_D)$ the right $\HH$-module of slice-regular functions on $\OO_D$.
When the domain $D$ intersects the real axis, this definition of slice regularity is equivalent to the one proposed by Gentili and Struppa \cite{GS}.
This approach to slice regularity has been developed on any real alternative *-algebra. See \cite{%CSS,
GhPe_AIM,GhPe_Trends, AlgebraSliceFunctions} for details and other references.

%%-- Numbered sections
\subsection{The slice derivatives and the spherical operators}

The commutative diagrams shown above suggest a natural definition of the \emph{slice derivatives} $\ddd{f}{x},\ddd{f\;}{x^c}$ of a slice functions $f$. They are the slice functions induced, respectively, by the derivatives $\ddd{F}{z}$ and $\ddd{F}{\overline z}$:
\[\dd{f}{x}=\I\left(\dd{F}{z}\right)\quad\text{and}\quad \dd{f\;}{x^c}=\I\left(\dd{F}{\overline z}\right).\]
With this notation a slice function is slice-regular if and only if $\ddd{f\;}{x^c}=0$ and if this is the case also the slice derivative $\ddd{f}{x}$ is slice-regular.
These derivatives satisfy the Leibniz formula for slice product of functions.

We now recall from \cite{GhPe_AIM} two other operators that describe completely slice functions.
Let $f=\I(F)$ be a slice function on $\OO_D$, induced by the stem function $F=F_1+\ui F_2$, with $F_1,F_2:D\subseteq\C\lra \HH$. 
\begin{definition}
The function $\vs f:\OO_D \lra \HH$, called \emph{spherical value} of $f$, and the function $f'_s:\OO_D \setminus \R \lra \HH$, called  \emph{spherical derivative} of $f$, are defined as
\[
\vs f(x):=\frac{1}{2}(f(x)+f(x^c))
\quad \text{and} \quad
f'_s(x):=\frac{1}{2}\IM(x)^{-1}(f(x)-f(x^c)).
\] 
\end{definition}

If $x=\alpha+\beta J\in\OO_D$ and $z=\alpha+\ui\beta\in D$, then $\vs f(x)=F_1(z)$ and $f'_s(x)=\beta^{-1} F_2(z)$. Therefore $\vs f$ and $f'_s$ are slice functions, constant on every set $\SS_x=\alpha+\beta\,\SS$. 
Observe that on $\OO_D\cap\R$, the spherical value of $f$ coincides with $f$. The functions $\vs f$ and $f'_s$
 are slice-regular only if $f$ is locally constant. Moreover, the formula
  \[\label{formula}
  f(x)=\vs f(x)+\IM(x)f'_s(x)
  \]
holds for each $x\in\OO_D\setminus \R$. If $F$ is of class $\mathcal{C}^1$, the formula holds also for $x\in\OO_D\cap\R$. In particular, if $f$ is slice-regular, $\sd{f}$ extends to the real points of $\OO_D$ with the values of the slice derivative $\ddd{f}{x}$.
The zero set $D_f$ of $\sd f$ is called \emph{degenerate set} of $f$ (see \cite[\S7]{GSS} for its properties).
The spherical value and the spherical derivatives satisfy the following Leibniz-type product rule (see \cite[\S5]{GhPe_AIM}):
\begin{equation}\label{Leibniz}
\sd {(f\cdot g)}=\sd f\cdot\vs g+\vs f\cdot\sd g.
\end{equation}

\subsection{Normal function and multiplicities of zeros}\label{subs:normal}
Given a slice function $f=\I(F):\OO_D\to\HH$, with $F=F_1+\ui F_2$, its \emph{conjugate function} $f^c$ and its \emph{normal function} $N(f)$ are the slice functions defined by
\[f^c=\I(F^c)=\I(F_1^c+\ui F_2^c)\text{\quad and\quad}N(f)=f\cdot f^c=f^c\cdot f,
\]
where $F_\mu^c(z)=\overline{F_\mu(z)}$ for $\mu=1,2$, $z\in D$.
The adjective \emph{normal} here is justified by the fact that $N(f)$ can be seen as the norm of $f$ in the *-algebra of slice functions with antiinvolution $f\mapsto f^c$ (in the literature, also the term \emph{symmetrization} is used for the normal function).
Observe that at every real point $a\in\OO_D\cap\R$, $f^c(a)=\overline{f(a)}$ and $N(f)(a)=|f(a)|^2$.
When $f$ is slice-regular, also $f^c$ and $N(f)$ are slice-regular, with $N(f)$ always slice-preserving (we refer to \cite[\S6]{GhPe_AIM} for more details about these functions). 

Let $V(f)=\{x\in\OO_D\,|\, f(x)=0\}$ be the set of zeros of the slice function $f$. We recall some of its basic properties (see \cite{GSS,GhPe_AIM}).  The elements $x\in V(f)$ can be of three types: \emph{real zeros} (when $x\in\R$), \emph{spherical zeros} (when $x\not\in\R$ and $\SS_x\subseteq V(f)$) or \emph{isolated nonreal zeros} (when $\SS_x\not\subseteq V(f)$). For any slice function $f$, it holds
\[V(N(f))=\bigcup_{y\in V(f)}\SS_y.\]
For every $f\in\SR(\OO_D)$, $f\not\equiv0$, the set $V(f)$  consists of isolated points (real or not real) or isolated 2-spheres of the form $\SS_x$ (with nonreal $x$).

\begin{definition}\label{def:nmult}
Let $f\in\SR(\OO_D)$, $f\not\equiv0$.
Let $\Delta_y(x)=N(x-y)$ 
be the \emph{characteristic polynomial} of $y\in\HH$.
Given $y\in V(f)$ and a non-negative integer $s$, we say that $y$ is a zero of $f$ of \emph{total multiplicity} $s$ if $\Delta_y^s$ divides $N(f)$ and $\Delta_y^{s+1}$ does not divide $N(f)$ in $\SR(\OO_D)$. We will denote the integer $s$ by $\widetilde m_f (y)$.
\end{definition}
Note that the total multiplicity is well-defined thanks to \cite[Corollary 23]{GhPe_AIM}. It has the property: $\widetilde m_{N(f)}(y)=2\widetilde m_f(y)$ for every $y\in V(f)$. This can be proved as in \cite[Theorem 26]{GhPe_AIM}, where the argument deals with slice-regular polynomials but it is valid for any slice-regular function (see also \cite[Proposition 6.14]{GSS}). 
A more refined definition of multiplicity for zeros of $f$ can be found in \cite[\S 3.6]{GSS}.

\subsection{Slice-regularity and harmonicity}

In this section we recall some results from \cite{Pe2018} concerning the relation between slice-regularity and harmonicity with respect to the standard Laplacian operator $\Delta_4$ of $\R^4$.
The Cauchy-Riemann-Fueter operator 
\[\dibar_{\CRF} =\dd{}{x_0}+i\dd{}{x_1}+j\dd{}{x_2}+k\dd{}{x_3}\]
 factorizes $\Delta_4$, since it holds:
\[\dcrf\dbcrf=\dbcrf\dcrf=\Delta_4,\]
where
\[\dcrf =\dd{}{x_0}-i\dd{}{x_1}-j\dd{}{x_2}-k\dd{}{x_3}\]
is the conjugated operator.
For any $i,j$ with $1\le i<j\le 3$, let $L_{ij}=x_i\ddd{}{x_j} -x_j\ddd{}{x_i}$ and let
$\Gamma=%-\sum_{i<j}e_ie_jL_{ij}=
-iL_{23}+jL_{13}-kL_{12}$
be the \emph{quaternionic spherical Dirac operator} on $\IM(\HH)$. The operators $L_{ij}$ are tangential differential operators for the spheres $\SS_x=\alpha+\beta\,\SS$ ($\beta>0$) and the  operator $\Gamma$ factorizes the Laplace-Beltrami operator of the 2-sphere.

\begin{proposition}(\cite[Proposition 6.1]{Pe2018})\label{propH}
Let $\OO=\OO_D$ be an open circular domain in $\HH$. For every slice function $f:\OO\lra\HH$, of class $\mathcal{C}^1(\OO)$, the following formulas hold on $\OO\setminus\R$:
\begin{itemize}\setlength\itemsep{0.5em}
\item[(a)]
$\Gamma f=2\IM(x)\sd{f}$.
\item[(b)]
$\dbcrf f-2\ddd{f}{x^c}=-2f'_s$.
\end{itemize}
\end{proposition}

\begin{proposition}(\cite[Corollary 6.2]{Pe2018})\label{pro:H}
Let $\OO=\OO_D$ be an open circular domain in $\HH$. Let  $f:\OO\lra\HH$ be a slice function of class $\mathcal{C}^1(\OO)$. Then
\begin{itemize}\setlength\itemsep{0.1em}
\item[(a)]
$f$ is slice-regular if and only if $\dbcrf f=-2f'_s$.
\item[(b)]
% $f$ is slice-regular and Fueter-regular (i.e.\ it belongs to the kernel of $\dbcrf$) if and only if $f$ is (locally) constant.
% % su ogni \emph{sfera} $\alpha+\SS^{n}\beta\subset \OO$ 
% \item[(c)]
$%\dcrf f-\dif f=
\dcrf f-2\ddd{f}{x}=2f'_s$ and
$%\dif f'_s=
2\ddd{\sd f}{x}=\dcrf f'_s$.
\end{itemize}
\end{proposition}

\begin{proposition}(\cite[Theorem 6.3]{Pe2018})\label{teo:sr}
Let $\OO=\OO_D$ be an open circular domain in $\HH$. 
If $f:\OO\lra\HH$ is slice-regular, then it holds:
\begin{itemize}\setlength\itemsep{0.1em}
\item[(a)]
  The spherical derivative $f'_s$ is harmonic on $\OO$ (i.e.\ its four real components are harmonic).
\item[(b)]
The following generalization of Fueter's Theorem holds: $\dbcrf\Delta_4f=0$.
As a consequence, every slice-regular function is biharmonic.
% \item[(c)]
% $\Delta_4 f=-4\,\dd{\sd f}{x}$. In particular, $\dd{\sd f}{x}$ is harmonic on $\OO$.
\end{itemize}
\end{proposition}

\section{A four dimensional Jensen formula}

In order to obtain the quaternionic version of Jensen formula, we need three preliminary results.
Before giving the statements, we clarify what we mean in the following by $\log|g|$, for a slice-preserving function $g=\I(G)=\I(G_1+\ui G_2)$ defined on $\OO$.
Since $G_1$ and $G_2$ are real-valued, the function $\overline g$ induced by the stem function $G_1-\ui G_2$ satisfies $\overline g(x)=\overline{g(x)}$ for every $x\in\OO$. 
Note that $(\overline g)'_s=-g'_s$, a property we will use later.
The function $|g|$ induced by the real-valued stem function $(G_1^2+G_2^2)^{1/2}$ satisfies $|g|(x)=|g(x)|$ for all $x\in\OO$. Moreover, $g\cdot \overline g=|g|^2$. Finally, the function $\log|g|=\I\left(\frac12 \log(G_1^2+G_2^2)\right)$ is a circular, slice-preserving function on $\OO\setminus V(g)$, satisfying $(\log|g|)(x)=\log|g(x)|$ for every $x\in\OO\setminus V(g)$.

The first result we need was proved in \cite[Theorem 2.1]{AltavillaBisi} using results from \cite{Pe2018}. For completeness we give a proof here. 

\begin{proposition}\label{pro:bih}
Let $\OO=\OO_D$ be an open circular domain in $\HH$. 
If $g:\OO\to\HH$ is slice-regular and slice-preserving, then $\ddd{}{x}\log|g|$ is slice-regular and $\log|g|$ is biharmonic 
on $\OO\setminus V(g)$. In particular, this is true when $g=N(f)$ for any slice-regular function $f:\OO\to\HH$.
% the function $\ddd{}{x}\log|N(f)|$ is slice-regular on $\OO\setminus V(N(f))$. As a consequence, $\log|N(f)|$ is biharmonic on $\OO\setminus V(N(f))$ for each $f\in\SR(\OO)$. 
\end{proposition}
\begin{proof}
Let $g=\I(G)=\I(G_1+\ui G_2)\in\SR(\OO)$ be slice-preserving. Let $\Delta_2$ be the two-dimensional Laplacian. Since $G:D\to \C$ is holomorphic, $\Delta_2\log|G|=0$ where $G$ does not vanish. Therefore $\ddd{}z \log|G|$ is holomorphic on $D\setminus V(G)$, and $\ddd{}{x}\log|g|=\I(\ddd{}z \log|G|)$ is slice-regular. 
Since $\log|g|$ is a circular slice function, its spherical derivative vanishes and then from point (b) of Proposition \ref{pro:H}
\[\textstyle\ddd{}{x}\log|g|=\dfrac12\dcrf\log|g|.\]
From point (b) of Proposition \ref{teo:sr} we get 
\[\textstyle
0=\dbcrf\Delta_4\left(\ddd{}x \log|g|\right)=\frac12\dbcrf\Delta_4\dcrf\log|g|=\frac12\Delta_4^2\log|g|,\]
i.e.\ $\log|g|$ is biharmonic.
\end{proof}

\begin{remark}
If $g:\OO\to\HH$ is slice-regular but not slice-preserving, then the function $\log|g|$, mapping $x$ to $\log|g(x)|$, can be not biharmonic. This can happen also if $g$ is one-slice-preserving (see \cite[Remark 2.8]{AltavillaBisi}). We recall that a slice function $f$ is \emph{one-slice-preserving} if there exists $J\in\SS$ such that $f(\OO\cap\C_J)\subseteq\C_J$.

\end{remark}

\begin{proposition}\label{deltalog}
Let $\OO$ be an open circular domain in $\HH$ with $\OO\ni0$. If $f:\OO\to\HH$ is slice-regular and $f(0)\ne0$, then
\[
\Delta_4\log|N(f)|_{x=0}=-4\RE\left(f(0)^{-1}\frac{\partial^2 f}{\partial x^2}(0)\right)+4\RE\left(\left(f(0)^{-1}\overline{\dd {f}{x}(0)}\right)^2\right).
\]
\end{proposition}
\begin{proof}
In this proof we denote  the spherical value and the spherical derivative of a slice function $f$ by $v_sf$ and $\partial_s f$ respectively. Let $g=\I(G)\in\SR(\OO)$ be slice-preserving. Using Propositions \ref{pro:H} and \ref{pro:bih}, we get, outside $V(g)$, 
\[\textstyle
\Delta_4(\log|g|^2)=\dbcrf\dcrf(\log|g|^2)=2\,\dbcrf\ddd{}x(\log|g|^2)=-4\,\partial _s\left(\ddd{}x(\log|g|^2)\right).
\] 
Since $g$ is slice-regular, $\overline g$ is anti-regular (i.e.\ in the kernel of $\ddd{}x$) and it holds, by the Leibniz rule for slice product, 
\[\textstyle
\ddd{}x(\log|g|^2)=\I\left(\ddd{}z(\log|G|^2)\right)=\I\left(\frac1{|G|^2}\ddd{}z(|G|^2)\right)=\frac1{|g|^2}\ddd{}x(g\cdot \bar g)
=\frac1{|g|^2}\ddd{g}x\cdot \bar g.
\]
Therefore
\[\textstyle
\Delta_4(\log|g|^2)=-4\,\partial _s\left(\frac1{|g|^2}\ddd{g}x\cdot \bar g\right).
\]
Since $|g|^{-2}$ is circular, from the Leibniz rule \eqref{Leibniz} for spherical value and derivative  we get
\begin{equation}\textstyle\label{pro:f0}
 \Delta_4(\log|g|^2)=-\frac4{|g|^2}\partial _s\left(\ddd{g}x\cdot \bar g\right)=
-\frac4{|g|^2}\left(\partial _s(\ddd{g}x)\cdot v_s{\bar g}+v_s{(\ddd{g}x)}\cdot \partial_s\bar g\right).
\end{equation}
Now we set $g=N(f)$. Firstly we must compute $\partial _s(\ddd{}x (f\cdot f^c))\cdot v_s(\overline{N(f)})$. Since $\ddd{}x (f\cdot f^c)=\ddd{f}x\cdot f^c+f\cdot \ddd {f^c}x$, we compute separately the two terms at $x=0$. The first one is
\[\textstyle
\left(\partial _s\left(\ddd{f}x\cdot f^c\right)\cdot v_s(\overline{N(f)})\right)_{|x=0}=
\left(\ddd{}x\left(\ddd{f}x\cdot f^c\right)\right)_{|x=0}\ |f(0)|^2,
\]
where we used the fact that the spherical derivative extends to $\R$ as the slice derivative. Then
\begin{equation}
\textstyle\label{pro:f1}
\left(\partial _s\left(\ddd{f}x\cdot f^c\right)\cdot v_s(\overline{N(f)})\right)_{|x=0}=
\left(\frac{\partial^2 f}{\partial x^2}(0)\,\overline{f(0)} + \left|\ddd{f}{x}(0)\right|^2\right) |f(0)|^2.
\end{equation}
The second term is
\begin{align}
\textstyle\nonumber\label{pro:f2}
\left(\partial _s\left(f\cdot \ddd {f^c}x\right)\cdot v_s(\overline{N(f)})\right)_{|x=0}&=\textstyle
\left(\ddd{}x\left(f\cdot \ddd {f^c}x\right)\right)_{|x=0} |f(0)|^2\\
&=\textstyle
\left(\left|\ddd{f}{x}(0)\right|^2 + f(0) \overline{\frac{\partial^2 f}{\partial x^2}(0)}\right) |f(0)|^2.
\end{align}
It remains to compute the two terms coming from $v_s\left(\ddd{}x(f\cdot f^c)\right)\cdot \partial_s\overline {N(f)}$.
Since $\partial_s\overline {N(f)}_{|x=0}=-\partial_s {N(f)}_{|x=0}=-\ddd {}x(f\cdot f^c)_{|x=0}$, the first one is equal to
\begin{equation}
\textstyle\label{pro:f3}
\left(v_s{(\ddd{f}x\cdot f^c)}\cdot \partial_s\overline {N(f)}\right)_{|x=0}=-2\ddd fx(0)\overline{f(0)}\,\RE\left(f(0)\overline{\ddd fx(0)}\right),
\end{equation}
whilst the second one is
\begin{equation}
\textstyle\label{pro:f4}
\left(v_s{(f\cdot \ddd{f^c}x)}\cdot \partial_s\overline {N(f)}\right)_{|x=0}=-2f(0)\overline{\ddd fx(0)}\,\RE\left(f(0)\overline{\ddd fx(0)}\right).
\end{equation}
Putting it all together, using  \eqref{pro:f1}, \eqref{pro:f2},  \eqref{pro:f3},  \eqref{pro:f4} in \eqref{pro:f0}, we get 
\begin{align*}
\textstyle
\Delta_4\log|N(f)|_{x=0}&=\frac12\Delta_4(\log|N(f)|^2)_{|x=0}\\
&\textstyle
=-\frac2{|f(0)|^2}\left( 2\RE\left(\overline{f(0)}\frac{\partial^2 f}{\partial x^2}(0)\right) + 2\left|\overline{\ddd {f}{x}(0)}\right|^2\right)\\
&\textstyle
\quad -\frac2{|f(0)|^4}\left(-4\left(\RE\left(f(0)\overline{\ddd {f}{x}(0)}\right)\right)^2 \right)\\
&\textstyle
=-4\RE\left(f(0)^{-1}\frac{\partial^2 f}{\partial x^2}(0)\right)+4\RE\left(\left(f(0)^{-1}\overline{\ddd {f}{x}(0)}\right)^2\right),
\end{align*}
where we used the fact that, for any $a\in\HH$, it holds $-2|a|^2+4\RE(a)^2=2\RE(a^2)$. 
\end{proof}

\begin{remark}
The formula of the previous proposition is still valid when $x=0$ is replaced by any real point $x$ where $f(x)\ne0$.
\end{remark}

Given a nonconstant function $f\in\SR(\OO)$, let $T_f:\OO\setminus V(N(f))\to\OO\setminus V(N(f))$ be the diffeomorphism defined by $T_f(x)=f^c(x)^{-1}x f^c(x)$ (see e.g. \cite[Proposition 5.32]{GSS}). Note that $T_f$ and its inverse $T_{f^c}$ map any sphere $\SS_x$ onto itself. 
Let $S_f:\OO\setminus V(N(f))\to\HH$ be the map  defined by
\[S_f(x)=\begin{cases}
\sd f(x)f(x)^{-1}\overline x f(x) \sd f(x)^{-1}&\text{\quad if }x\not\in \overline{D_f},\\
\overline x&\text{\quad if }x\in \overline{D_f}.
\end{cases}
\]
Observe that if $f$ is slice-preserving, then $S_f(x)=\overline x$ for every $x$.

\begin{proposition}\label{Sf}
Let $f\in\SR(\OO)$ be nonconstant. 
The map $S_f$ is a diffeomorphism of the open set $\OO\setminus (V(N(f))\cup \overline{D_f})$.
\end{proposition}
\begin{proof}
Let $x\in\OO\setminus (V(N(f))\cup \overline{D_f})$ and $y=S_f(x)$. Since $y\in\SS_x$, we have $\sd f(y)=\sd f(x)$ and then 
\[
\sd f(y)^{-1}y\sd f(y)=f(x)^{-1}\overline x f(x).
\] 
Since 
\[
\overline{f(x)}\,x\,\overline{f(x)}^{-1}=n(f(x))f(x)^{-1}x\frac{f(x)}{n(f(x))}=T_{f^c}(x),
\]
 it holds $T_{f^c}(x)=\overline{\sd f(y)^{-1}y\sd f(y)}$. Therefore $x=T_f\left(\overline{\sd f(y)^{-1}y\sd f(y)}\right)$ and the map $y\mapsto T_f\left(\overline{\sd f(y)^{-1}y\sd f(y)}\right)$ is the inverse of $S_f$ on $\OO\setminus (V(N(f))\cup \overline{D_f})$.
\end{proof}

Let $\BB(0,r)$ be an open ball with centre 0 and radius $r$ with closure contained in $\OO$ and let $f\in\SR(\OO)$ be nonconstant. Assume that $f\ne0$ on $\partial\BB(0,r)$. Under this condition, the map $S_f$ applies $\partial\BB(0,r)$ onto itself.
 From what recalled in \S\ref{subs:normal}, the zero set $V(f)\cap\BB(0,r)$ consists of a finite number of isolated real points $r_1,\ldots,r_m$, of isolated spheres $\SS_{x_1},\ldots, \SS_{x_t}$ and isolated nonreal points $a_{t+1},\ldots,a_l$. We choose one spherical zero $a_i$ in every sphere $\SS_{x_i}$, for $i=1,\ldots,t$.

We are now able to state the four-dimensional Jensen formula for slice-regular functions. 

\begin{theorem}\label{teo:Jensen}
Let $\OO$ be an open circular domain in $\HH$. Let $\BB_r=\BB(0,r)$ be an open ball whose closure is contained in $\OO$.
If $f:\OO\to\HH$ is slice-regular and not constant,  $f(0)\ne0$ and $f\ne0$ on $\partial\BB_r$, then it holds:
\begin{align*}
&\log|f(0)|+\frac{r^2}4\RE\left(\left(f(0)^{-1}\overline{\dd {f}{x}(0)}\right)^2\right)-\frac{r^2}4\RE\left(f(0)^{-1}\frac{\partial^2 f}{\partial x^2}(0)\right)=\\&=
\frac1{2|\partial\BB_r|}\int_{\partial\BB_r}\log|f(x)|d\sigma(x)+
\frac1{2|\partial\BB_r|}\int_{\partial\BB_r}\log|f\circ S_f(x)|d\sigma(x)\nonumber \\
&\quad
-\sum_{k=1}^m\left(\log\frac{r}{r_k}+\frac{r_k^4-r^4}{4\,r^2r_k^2}\right) \nonumber
 -\sum_{i=1}^l\left(2\log\frac{r}{|a_i|}+\frac{|a_i|^4-r^4}{4\,r^2 |a_i|^4}\left(t(a_i)^2-2|a_i|^2\right)\right) \nonumber
\end{align*}
where the first sum ranges over the real zeros $r_1,\ldots, r_m$ of $f$ in $\BB_r $ and the second one over the non-real 
zeros $a_1,\ldots,a_l$ of $f$ in $\BB_r$, repeated according to their total multiplicities.  
\end{theorem}

\begin{proof}
Let $x=\alpha+J\beta\in\partial\BB_r\setminus\overline{D_f}$ and let $z=\alpha+\ui\beta$. Since $f(x)\ne0$,  $N(f)(x)=f(x)f^c(T_{f^c}(x))$. Moreover, if $f=\I(F)=\I(F_1+\ui F_2)$ and $J'=f(x)^{-1}Jf(x)$, then $T_{f^c}(x)=f(x)^{-1}xf(x)=\alpha+J'\beta$ and 
\[f^c(T_{f^c}(x))=\overline{F_1(z)}+J'\overline{F_2(z)}=\overline{F_1(z)-F_2(z)J'}=\overline{F_1(z)+KF_2(z)},
\]
where $K=-F_2(z)J'F_2(z)^{-1}=-\sd f(x)J'\sd f(x)^{-1}\in\SS$. Therefore $f^c(T_{f^c}(x))=\overline{f(S_f(x))}$ and then
\begin{equation*}\label{eq:lognf0}
\log|N(f)(x)|=\log|f(x)|+\log|f(S_f(x))|\text{\quad on }\partial\BB_r\setminus\overline{D_f}.
\end{equation*}
On the other hand, if $x\in\partial\BB_r\cap\overline{D_f}$ then $N(f)(x)=f(x)f^c(x)=f(x)\overline{f(x)}$ and $|\overline{f(x)}|=|f(x)|=|f(\bar x)|$. Therefore \begin{equation}\label{eq:lognf}
\log|N(f)(x)|=\log|f(x)|+\log|f(S_f(x))|\text{\quad on }\partial\BB_r.
\end{equation}
The Jensen formula for $f$ follows now from the formula proved in \cite[Theorem~3.3]{AltavillaBisi} applied to the slice-preserving regular function $N(f)$, using equation \eqref{eq:lognf}, the formula $N(f)(0)=|f(0)|^2$, Proposition \ref{deltalog} and the fact that the total multiplicities of zeros of $f$ are one half the total multiplicities of them as zeros of $N(f)$.
\end{proof}

\begin{remark}
If $f$ has no zeros in $\overline{\BB}_r$, the previous Jensen formula is a consequence of the mean value formula for biharmonic functions applied to $\log|N(f)|$. In this case the last two sums in the formula are missing.
\end{remark}

The Jensen formula can be extended to  
\emph{semiregular functions}, the analogues of meromorphic functions in the quaternionic setting (see \cite[\S 5]{GSS} and \cite{Singularities} for definitions and properties of these functions). In the slice-preserving case, Jensen formula formula for semiregular functions was proved in \cite{AltavillaBisi}.

Let $\OO$ be an open circular domain in $\HH$ and let $\BB_r=\BB(0,r)$ be an open ball whose closure is contained in $\OO$.
Let $f:\OO\to\HH$ be semiregular. We denote by $r_1,\ldots r_m$ the real zeros of $f$ in $\overline{\BB}_r$, by $a_1,\ldots, a_l$  the non-real zeros of $f$ in $\overline{\BB}_r$ (as above, in case of spherical zeros we choose one spherical zero in every sphere), repeated according to their total multiplicities. 

The poles of $f$ can be real or spherical. In the latter case, if $\SS_x$ is a spherical pole, the \emph{order} $\ord_f(y)$ of the points $y\in\SS_y$ are all equal, except possibly for one point of lesser order (see \cite[Theorem 5.28]{GSS} and \cite[Theorem 9.4]{Singularities}). 
We denote by $p_1,\ldots, p_n$ the real poles of $f$ in $\overline{\BB}_r$, repeated according to their order. 
Let $\SS_{y_1},\ldots,\SS_{y_p}$ be the spherical poles of $f$ in $\overline{\BB}_r$ of the first type, having the property that all points in $\SS_{y_i}$ have the same order. Let $\SS_{z_1},\ldots,\SS_{z_q}$ be the spherical poles of $f$ in $\overline{\BB}_r$ of the second type, with the points $z_j\in\SS_{z_j}$ chosen such that $\ord_f(z_j)<\max_{z\in\SS_{z_j}}\ord_f(z)$. Let $i_f(z_j)>0$ denote the \emph{isolated multiplicity} of $f$ at $z_j$ for $j=1,\ldots,q$, as in \cite[Definition 3.12]{Stoppato2012}. 
Set \[s_1=\frac12\sum_{i=1}^p\ord_f(\SS_{y_i}),\quad s_2=\frac12\sum_{j=1}^q\ord_f(\SS_{z_j}),\quad s=s_1+s_2,\]
where $\ord_f(\SS_x)$ is the $\emph{spherical order}$ of $f$ at $\SS_x$ (which is two times the maximal order of the points of the sphere \cite[Theorem 9.4]{Singularities}). 
It holds $i_f(z_j)\ge\frac12\ord_f(\SS_{z_j})-\ord_f(z_j)>0$ for every $j=1,\ldots,q$ (\cite[Proposition 5.31]{GSS}).

Choose points $b_1,\ldots,b_{s_1}\in\cup_{i=1}^p\SS_{y_i}$ and $b_{s_1+1},\ldots,b_s\in\cup_{j=1}^q\SS_{z_j}$ (one point in each sphere, repeated according to one-half the spherical order of the pole).
Let $a_{l+1},\ldots,a_{l+q'}$ denote the points $z_1,\ldots,z_q$, repeated according to their isolated multiplicities ($q'=\sum_{j=1}^q i_f(z_j)$).

With these notations, we can state the Jensen formula for semiregular functions.

\begin{theorem}\label{teo:Jensensemi}
Let $\OO$ be an open circular domain in $\HH$ and let $\BB_r=\BB(0,r)$ be an open ball whose closure is contained in $\OO$.
Let $f:\OO\to\HH$ be semiregular and not constant. %, such that on every spherical pole of $f$ all points have the same order. 
Assume that $0$ is not a pole nor a zero of $f$ and $\partial\BB_r$ does not contain zeros or poles of $f$. Then it holds:
\begin{align*}
&\log|f(0)|+\frac{r^2}4\RE\left(\left(f(0)^{-1}\overline{\dd {f}{x}(0)}\right)^2\right)-\frac{r^2}4\RE\left(f(0)^{-1}\frac{\partial^2 f}{\partial x^2}(0)\right)=
\\
&=
\frac1{2|\partial\BB_r|}\int_{\partial\BB_r}\log|f(x)|d\sigma(x)+
\frac1{2|\partial\BB_r|}\int_{\partial\BB_r}\log|f\circ S_f(x)|d\sigma(x)\nonumber
\\
&\quad
-\sum_{k=1}^m\left(\log\frac{r}{r_k}+\frac{r_k^4-r^4}{4\,r^2r_k^2}\right) \nonumber
 -\sum_{i=1}^{l+q'}\left(2\log\frac{r}{|a_i|}+\frac{|a_i|^4-r^4}{4\,r^2 |a_i|^4}\left(t(a_i)^2-2|a_i|^2\right)\right)\\ 
 \nonumber
&\quad
 +\sum_{k=1}^n\left(\log\frac{r}{p_k}+\frac{p_k^4-r^4}{4\,r^2p_k^2}\right) \nonumber
 +\sum_{i=1}^{s}\left(2\log\frac{r}{|b_i|}+\frac{|b_i|^4-r^4}{4\,r^2 |b_i|^4}\left(t(b_i)^2-2|b_i|^2\right)\right). \nonumber
\end{align*}
\end{theorem}
\begin{proof}
The proof is based on the fact that one can find a slice-preserving regular function $g$ on an open neighbourhood $\OO'$ of $\overline\BB_r$ such that $gf$ has a slice-regular extension $h$ on $\OO'$. For every real pole $p_k\in\BB_r$, let
\[g_{k}^{(1)}(x)=-(x-r^2 p_k^{-1})^{-1}(x-p_k)rp_k^{-1}
\]
and for every spherical pole $b_i$, let
\[g^{(2)}_i(x)=
\Delta_{r^2b_i^{-1}}(x)^{-1}\Delta_{b_i}(x)r^2|b_i|^{-2}.
\]
Observe that $g_k^{(1)}$ is the reciprocal of the slice-preserving quaternionic $r$-Blaschke factor $B_{p_k,r}$
and $g_i^{(2)}(x)$ is the reciprocal of the normal function $N(B_{b_i,r})$ (see e.g.\ \cite{AltavillaBisi} for definition and properties of quaternionic $r$-Blaschke factors $B_{a,r}$). We can set 
\[g=g^{(1)}_1\cdots g^{(1)}_n g^{(2)}_1\cdots g^{(2)}_s.
\]
Then $g$ is a slice-preserving regular function on a neighbourhood $\OO'$ of $\overline{\BB}_r$, such that $|g|=1$ on $\partial\BB_r$,  
having zero set $V(g)=\{p_1,\ldots,p_n\}\cup\{\SS_{b_1},\ldots,\SS_{b_s}\}$ (with multiplicities equal to the orders of the poles for $f$). We can assume that all the zeros and poles of $f$ stay in $\OO'$. Then $gf$ extends regularly to a function $h\in\SR(\OO')$. 

If the order of $f$ is constant on every spherical pole ($q=q'=0$), then $V(h)\cap\BB_r=V(f)\cap\BB_r$ with equal total multiplicities. The Jensen formula for $f$ now follows from the formula for $h$, using the following facts:
\begin{enumerate}
	\item 
	$|h(0)|=|f(0)||g(0)|=|f(0)|\prod_{k=1}^n\frac{|p_k|}r\prod_{i=1}^s\frac{|b_i|^2}{r^2}$.
	\item
	On $\partial\BB_r$, since $|g|=1$ and $g$ is slice-preserving, it holds 	$|h|=|f|$ and
	$\log|N(h)|=2\log|g|+\log|N(f)|=\log|N(f)|$.
	\item\label{pt4}
	From \cite[Lemma 3.1]{AltavillaBisi} (or from Proposition \ref{deltalog}), it follows that
	\begin{align*}
	\Delta_4&\log|N(h)|_{x=0}=\Delta_4\log|N(f)|_{x=0}+\Delta_4\log|N(g)|_{x=0}=\\
	&=\Delta_4\log|N(f)|_{x=0}-2\sum_{k=1}^n\frac{p_k^4-r^4}{r^4p_k^2}-2\sum_{i=1}^s\frac{|b_i|^4-r^4}{r^4 |b_i|^4}\left(t(b_i)^2-2|b_i|^2\right).
	\end{align*}
\end{enumerate}

If the order of $f$ varies on some spherical pole $\SS_{b_j}$, 
then $V(h)$ vanishes also at the points $z_1,\ldots, z_q$, with total multiplicities equal to the isolated multiplicities $i_f(z_j)$.
The Jensen formula for $f$ follows from the formula for $h$, using properties 1, 2, 3 above and the equality
\begin{align*}
& \sum_{j=1}^q\left(2\log\frac{r}{|z_j|}+\frac{|z_j|^4-r^4}{4\,r^2 |z_j|^4}\left(t(z_j)^2-2|z_j|^2\right)\right)i_f(z_j)\\
& \quad=\sum_{i=l+1}^{l+q'}\left(2\log\frac{r}{|a_i|}+\frac{|a_i|^4-r^4}{4\,r^2 |a_i|^4}\left(t(a_i)^2-2|a_i|^2\right)\right).
\end{align*}
\end{proof}

\begin{remark}
An example of semiregular function that has a spherical pole where the order is not constant, is given
by $f(x)=(x^2+1)^{-1}(x+i)$. It has no zeros and one spherical pole at $\SS=\SS_i$, whose points have all order 1, except for $x=-i$, that has order 0 and isolated multiplicity 1. One obtains the Jensen formula for $f$ on $\BB_r$ ($r>1)$ by multiplying $f$ on the left by the slice-preserving function
\[g=g^{(2)}_1=r^2(x^2+r^4)^{-1}(x^2+1).
\]
and applying Theorem \ref{teo:Jensen} to the product $h=gf=r^2(x^2+r^4)^{-1}(x+i)$, which is slice-regular on $\HH\setminus\SS_{r^2i}$, a neighbourhood of $\BB_r$.
This example shows that the contribution to the formula of spherical poles with nonconstant order can cancel out. 
This happens  when  $i_f(z_j)=\frac12\ord_f(\SS_{z_j})$ for every $j=1,\ldots,q$. 
\end{remark}

%-----------------------------------------------
% Authors (name, address, e-mail)
%-----------------------------------------------
\bigskip
\bigskip
\begin{minipage}[t]{10cm}
\begin{flushleft}
\small{
\textsc{Alessandro Perotti}
% \address{Department of Mathematics\\
%     University of Trento\\
%     Via Sommarive 14\\
%     I-38123 Trento\\
%     Italy}
% \email{alessandro.perotti@unitn.it}
\\*Department of Mathematics
\\*University of Trento
\\*Via Sommarive 14
\\* Trento, I-38123, Italy
\\*e-mail: alessandro.perotti@unitn.it
}
\end{flushleft}
\end{minipage}

%---------------------------------------------------- THE END
\end{document}